\documentclass{amsart}
\usepackage{amssymb}
\usepackage{amsfonts}

\newtheorem{theorem}{Theorem}
\theoremstyle{plain}

\newtheorem{corollary}{Corollary}

\newtheorem{definition}{Definition}
\newtheorem{example}{Example}

\newtheorem{proposition}{Proposition}
\newtheorem{remark}{Remark}

\numberwithin{equation}{section}

\begin{document}

\title[$\phi${-}$\delta$-Primary {Hyperideals in Krasner Hyperrings }]{$\phi$-$\delta$-Primary Hyperideals in Krasner Hyperrings }
\author[E. Kaya]{Elif Kaya}
\address{Department of Mathematics and Science Education, Istanbul Sabahattin Zaim
University, Istanbul, Turkey. Orcid: 0000-0002-8733-2934}
\email{elif.kaya@izu.edu.tr}
\author[M. Bolat]{Melis Bolat}
\address{Department of Mathematics, Yildiz Technical University, Istanbul, Turkey.
Orcid: 0000-0002-0667-524X}
\email{melisbolat1@gmail.com}
\author[S. Onar]{Serkan Onar}
\address{Department of Mathematical Engineering, Yildiz Technical University, Istanbul,
Turkey. Orcid: 0000-0003-3084-7694}
\email{serkan10ar@gmail.com}
\author[B.A. Ersoy]{Bayram Ali Ersoy}
\address{Department of Mathematics, Yildiz Technical University, Istanbul, Turkey.
Orcid: 0000-0002-8307-9644}
\email{ersoya@yildiz.edu.tr}
\author[K. Hila]{Kostaq Hila}
\address{Department of Mathematical Engineering, Polytechnic University of Tirana,
Tirana, Albania. Orcid: 0000-0001-6425-2619}
\email{kostaq\_hila@yahoo.com}

\begin{abstract}
In this paper, we study commutative Krasner hyperring with nonzero
identity. $\phi$-prime, $\phi$-primary and $\phi$-$\delta$-primary hyperideals
are introduced. The concept of $\delta$-primary
hyperideals is extended to $\phi$-$\delta$-primary hyperideals. Some
characterizations of hyperideals are provided to classify them. The relation between $\phi$-$\delta$-primary hyperideal and other hyperideals is discussed.

\end{abstract}
\maketitle

\section{Introduction}

In commutative ring theory prime and primary ideals have a significant place.
The importance of prime ideals encourages researchers to expand these concepts
and find applications. Many different types of generalizations have been
investigated by several authors, some of them \cite{2}, \cite{5}, \cite{6},
\cite{9}, \cite{12}. Prime and primary ideals are generalized to $\phi
$-prime and $\phi$-primary ideals. Let $\Re$ be a commutative ring
with nonzero identity. Denote the set of all ideals of $\Re$ by $L(\Re)$
(proper ideals of $\Re$ by $L^{\ast}(\Re).$ Let $\phi$ be a function such that
$\phi:L(\Re)\rightarrow L(\Re)\cup\{\emptyset\}$. Let $N$ be an ideal of $\Re$.
$N$ is called a $\phi$-prime ideal in \cite{2}, if $ab\in N-\phi(N)$, then
either $a\in N$ or $b\in N$ for some $a,b\in\Re$. By the way, $N$ is called a
$\phi$-primary ideal when, if $ab\in N-\phi(N)$, then $a\in N$ or $b^{k}\in N$ for
some $a,b\in\Re,~k\in%
\mathbb{N}
$ \cite{6,9}. The image of ideal, $\phi(N),$ can be equal $0,$
$\emptyset,~$\ $N,~N^{2},$ $N^{n},$ $N^{w}$ ($w$ denotes the intersection of
ideals of $N_{i}$). A proper ideal $N$ of $\Re$ is called weakly prime
(primary) ideal respectively in \cite{3} (\cite{4}) if $0\neq ab\in N,$ for
some $a,b\in\Re$, then $a\in N$ or $b\in N$ ($b^{k}\in N$ for some $k\in%
\mathbb{N}
$). Anderson generalized it in \cite{2}, where $N$ is weakly $\phi$-prime ideal when
$\phi(N)=0.$ Zhao introduced $\delta$-primary ideal as an expansion of an
ideal, in \cite{19}, $\delta:L(\Re)\longrightarrow L(\Re)$ is a function that
meets the following requirements: $i)$ $N$ $\subseteq\delta(N)$, for all
ideals $N$ of $\Re,ii)$ If $N\subseteq M,$ where $N$ and $M$ are ideals of
$\Re$, then $\delta(N)\subseteq\delta(M),~iii)$ $\delta(K\cap L)=\delta
(K)\cap\delta(L)$\ for all ideals $K,L$ of $\Re.$ Entire of the $\delta$ ideal
expansions provides the property $\delta^{2}=\delta$, which is $\delta
(\delta(N))=\delta(N)$ for all ideal $N$ of $\Re$ \cite{19}. A. Jaber chose
$\phi$ such a reduction function in \cite{13}, which satisfies the following
requirements: $i)$ $\phi(N)\subseteq N$, for all ideals $N$ of $\Re, $ $ii)$
If $N\subseteq M,$ where $N$ and $M$ are ideals of $\Re$, then $\phi
(N)\subseteq\phi(M).$ He obtained generalization of $\phi$-$\delta
$-primary ideal by combining these two concepts. Let $N$ be an ideal of
$\Re, \delta$ be an ideal expansion and $\phi$ be an ideal reduction \cite{13}. $N$ is called $\phi$-$\delta$-primary if $ab\in N-$ $\phi(N)$, then
either $a\in N$ or $b\in\delta(N)$, for all $a,b\in\Re.$

The theory of hyperstructures was innovated by Marty in 1934 \cite{15}. He
defined hypergroupoid $(G,\circ)$ for $G\not =\emptyset$, $P^{\ast}(G)$
represents family of nonempty subsets of $G$ and $\circ:G\times
G\longrightarrow P^{\ast}(G)$ is a hyperoperation. Let $(G,\circ)$ be a
hypergroupoid. $G$ is a semihypergroup, if $\forall a,b,c\in G$, $a\circ(b\circ c)=(a\circ
b)\circ c,$ which means $\bigcup\limits_{u\in a\circ b}u\circ c=\bigcup
\limits_{v\in b\circ c}a\circ v$. If $\forall a\in G,~$ there exists $e\in G$ such that $a\in(e\circ a)\cap(a\circ
e)$ in another phrase $\left\{  a\right\}  =(e\circ a)\cap(a\circ e),$ then
$e$ is called identity element. Let $(G,\circ)$ be a semihypergroup. For
$\forall a\in G,$ if $a\circ G=G\circ a=G,$ then $(G,\circ)$ is called
hypergroup. Let $(G,\circ)$ be a hypergroup and $\emptyset\neq K$ be a subset
of $G.~$If $a\circ K=K\circ a=K,$ for $\forall a\in K,$ then $(K,\circ)$ is
called subhypergroup of $(G,\circ)$. Let $(G,\circ)$ be a hypergroup.
If $a\circ b=b\circ a,$ for $\forall a,b\in G,$ then $(G,\circ)$ is
commutative hypergroup \cite{15}. Mittas pioneered the theory of canonical
hypergroups in \cite{16}: Let $\Re\not =\emptyset$. $(\Re,+)$ is called a
canonical hypergroup ($+$ is a hyperoperation) if the following
axioms are satisfied: $i)$ $a+(b+c)=(a+b)+c,$ for $a,b,c\in\Re; ii)$ $a+b=b+a,$ for
$a,b\in\Re;$ $iii)$ $\exists0\in\Re$ such that $a+0=\{a\},$ for any $a\in\Re;$
$iv)$ for any $a\in\Re,$ $\exists!a^{\prime}\in\Re,$ such that $0\in
a+a^{\prime}$; $v)$ $c\in a+b$ implies that $b\in-a+c$ and $a\in c-b,$ which
means $(\Re,+)$ is reversible. Hyperrings and hyperfields were introduced by
Krasner in \cite{14} using the canocical hypergroups. $(\Re,+,\cdot)$ is called
Krasner hyperring if the following statements hold: $i) (\Re,+)$ is a
canonical hypergroup; $ii) (\Re,\cdot)$ is a semigroup having $0$ as $a\cdot 0=0\cdot a=0,$
for all $a\in\Re;$ $iii)$ $(b+c)\cdot a=(b\cdot a)+(c\cdot a)$ and $a\cdot (b+c)=(a\cdot b)+(a\cdot c),$ for
all $a,b,c\in\Re.$ Also Ameri and Norouzi, Corsini, Dasgupta, Davvaz, Leoreanu
Fotea, studied hyperrings in more details in \cite{1}, \cite{7}, \cite{8},
\cite{10}, \cite{11} \cite{17}. Ye\c{s}ilot et al. defined a hyperideal
expansion and $\delta$-primary hyperideal in \cite{18}. $\delta:L(\Re
)\longrightarrow L(\Re)$ function is defined as a hyperideal expansion
function that meets the following requirements: $i)$ $N$ $\subseteq\delta(N)$,
for all hyperideals $N$ of $\Re,$ $ii)$ If $N\subseteq M,$ where $N$ and $M$
are hyperideals of $\Re$, then $\delta(N)\subseteq\delta(M).$ Given an
expansion $\delta$ of hyperideals $N$ of a hyperring $\Re$ is called $\delta
$-primary if $ab\in N$, then $a\in N$ or $b\in\delta(N)$, for all $a,b\in\Re.$

In this paper, our goal is to extend the concept of $\delta$-primary
hyperideals to $\phi$-$\delta$-primary hyperideals in Krasner hyperrings and
we intend to give some generalizations of hyperideals. Throughout this paper,
$(\Re,\oplus,\circ)$ will be a commutative Krasner hyperring with nonzero
identity. We denote the set of all hyperideals of $\Re$ by $L(\Re)$ and the
set of all proper hyperideals of $\Re$ by $L^{\ast}(\Re).$ We take $\phi$ as a
function $\phi:L(\Re)\rightarrow L(\Re)\cup\{\emptyset\}$ at the sections of
generalizations of prime and primary hyperideals$.$ Firstly, we define $\phi
$-prime and $\phi$-primary hyperideals in Krasner hyperrings and we give some
characterizations for prime and primary hyperideals. Let $\phi$ be a function
such that $\phi:L(\Re)\rightarrow L(\Re)\cup\{\emptyset\}$ and $N$ be a
hyperideal of $\Re.$ $N$ is said to be $\phi$-prime ($\phi$-primary)
hyperideal of $\Re$ if $a\circ b\in N-$ $\phi(N),$ then $a\in N$ or $b\in N$
$(b^{n}\in N$ respectively, $\exists n\in%
\mathbb{N}
),$ for $a,b\in\Re.$ Among other things, we give some characterizations in
Theorem \ref{Thmchar} and Theorem \ref{Thmchar2} as main theorems. Then we
define $\phi$-$\delta$-primary hyperideals in Krasner hyperrings and also give
several characterizations (See; Theorem \ref{tm1} and Theorem \ref{tweak}). In
this case, difference is that, $\phi:L(\Re)\rightarrow L(\Re)\cup
\{\emptyset\}$ is a reduction function if $\phi(N)\subseteq N\ $and
$N\subseteq M\ $implies $\phi(N)\subseteq\phi(M)\ $for each $N,M\in L(\Re)$.
We\ investigate the attitude of $\phi$-$\delta$-primary hyperideal under
homomorphism, in quotient ring, in cartesian product and other cases (See;
Theorem \ref{thom}, Proposition \ref{propcart}, Proposition \ref{ploc},
Theorem \ref{thcart}).

\section{{ Generalizations of Prime Hyperideals in Krasner Hyperrings}}

Throughout this section, $(\Re,\oplus,\circ)$ is a commutative Krasner
hyperring with nonzero identity. We denote the set of all hyperideals of $\Re$ by
$L(\Re)$. Initially, we give the definition of $\phi$-prime hyperideal and
some examples.

\begin{definition}
Let $\phi$ be a function such that $\phi:L(\Re)\rightarrow L(\Re
)\cup\{\emptyset\}$ and $N$ be a hyperideal of $\Re.$ $N$ is said to be $\phi
$-prime hyperideal of $\Re$ if $a\circ b\in N-$ $\phi(N),$ then $a\in N$ or
$b\in N,$ for $a,b\in\Re.$
\end{definition}

\begin{example}
\label{exam1} Let $\Re$ be a commutative Krasner hyperring. Consider the
following functions $\phi:L(\Re)\rightarrow L(\Re)\cup\{\emptyset\}$ defined
as follows: for any $N\in L(\Re):$

(i) $\phi_{\oslash}(N)=\emptyset$.

(ii) $\phi_{0}(N)=0.$

(iii) $\phi_{2}(N)=N^{2}.$

(iv) $\phi_{n}(N)=N^{n},$ (for any $n\geq2).\ \ \qquad$

(v) $\phi_{w}(N)=\cap_{n=1}^{\infty}N^{n}.\qquad$

(vi) $\phi_{1}(N)=N.$
\end{example}

It is obvious that $\phi_{\oslash}\leq\phi_{0}\leq\phi_{w}\leq...\leq
\phi_{n+1}\leq\phi_{n}\leq\phi_{n-1}\leq...\leq\phi_{2}\leq\phi_{1}.$

\begin{example}
\label{ex1prime} Let $\Re$ be a hyperring and $N\ $be a proper hyperideal of
$\Re. $

(i) $N$ is prime hyperideal if and only if it$\ $is $\phi_{\emptyset}$-prime hyperideal.

(ii) $N $is weakly prime hyperideal if and only if it$\ $is $\phi_{0}$-prime hyperideal.

(iii) $N$ is almost prime hyperideal if and only if it$\ $is $\phi_{2}$-prime hyperideal.

(iv) $N$ is $n$-almost prime hyperideal if and only if it$\ $is $\phi_{n}%
$-prime hyperideal.

(v)\ $N$ is $w$-prime hyperideal if and only if it is $\phi_{w}$-prime hyperideal.
\end{example}

\begin{proposition}
Let $N$ be a proper hyperideal of $\Re.$

$(1)$ If $\sigma_{1},$ $\sigma_{2}$\ are two functions with $\sigma_{1}\leq$
$\sigma_{2}$ such that $\sigma_{1},\sigma_{2}:L(\Re)\rightarrow L(\Re
)\cup\{\emptyset\}$ and $N$ is $\sigma_{1}$-prime, then $N$ is $\sigma_{2}$-prime.

$(2)(i)$ $N$ is prime hyperideal $\Longrightarrow$ $N$\ is weakly prime
hyperideal $\Longrightarrow$ $N$ is $w$-prime hyperideal $\Longrightarrow$ $N$
is $(n+1)$-almost prime hyperideal $\Longrightarrow$ $N$ is $n$-almost prime
hyperideal for any $n\geq2\Longrightarrow$ $N$ is almost prime hyperideal.

$(ii)$ $N$ is $w$-prime hyperideal if and only if $N$ is $n$-almost prime for
any $n\geq2.$
\end{proposition}

\begin{proof}
$(1)$ Let us assume $N$ is a $\sigma_{1}$-prime hyperideal of $\Re$. Take $a\circ b\in N$ and $a\circ b\notin\sigma_{2}(N),$ for $a,b\in\Re.$ It follows $a\circ b\notin\sigma_{1}(N)$ because of $\sigma_{1}\leq$ $\sigma_{2}.$
Therefore $a\circ b\in N-\sigma_{1}(N)$ and on the assumption that $N$ is
$\sigma_{1}$-prime hyperideal, then $a\in N$ or $b\in N$. That means $N$ is
$\sigma_{2}$-prime hyperideal of $\Re$.

$(2)(i)$ We can obtain it by the ordering of the $\phi_{i}^{\prime}$s in
Example \ref{exam1} with the proof of 1.

$(ii)$ Assume that $N$ is $w$-prime. Let $a\circ b\in N-\phi_{w}%
(N)=N-\cap_{n=1}^{\infty}N^{n}$. Then $a\circ b\in N-\phi_{n}(N)$, since
$\phi_{w}\leq\phi_{n},$ from (1) $N$ is $\phi_{n}$-prime. It means $N$ is
$n$-almost prime, for all $n\geq2.$ Conversely, suppose that $N$ is $n$-almost
prime, for $a\circ b\in N-N^{n}$ for all $n\geq2$. Thus $a\circ b\in
N-\cap_{n=2}^{\infty}N^{n}.$ Hence $a\circ b\in N-\cap_{n=1}^{\infty}N^{n}.$
Therefore $a\in N$ or $b\in N$, since $N$ is $n$-almost prime.
\end{proof}

\begin{theorem}
\label{thm5} Let $\phi:L(\Re)\rightarrow L(\Re)\cup\{\emptyset\}$ be a
function and $T~$be a proper hyperideal of $\Re$ such that $T~$ is a $ \phi$-prime
hyperideal of $\Re$. If $T$ is not prime, then $T^{2}\subseteq$ $\phi(T).$
Hence at the same time it means if $T^{2}\nsubseteq$ $\phi(T),$ then $T$ is prime.
\end{theorem}

\begin{proof}
Assume that $T^{2}\nsubseteq$ $\phi(T).~$We need to prove that $T$ is prime.
Take $x\circ y\in T$ for $x,y\in\Re.$ If $x\circ y\notin\phi(T),$ then $x\circ
y\in T-\phi(T).$ Since $T$ is $\phi$-prime, then $x\in T$ or $y\in T.$ Let us suppose
$x\circ y\in\phi(T)$. We can presume $x\circ T\nsubseteq$ $\phi(K)$, so
$x\circ m\notin\phi(T)$ for some $m\in T.$ Then $x\circ(y\oplus m)\subseteq
T-\phi(T).$ So $x\in T$ or $y\oplus m\subseteq T$, since $T~$is $\phi$-prime
hyperideal and hence $x\in T$ or $y\in T.$ Now on we can assume that $x\circ
T\subseteq$ $\phi(T).$ (Similarly, $y\circ T\subseteq$ $\phi(T).)$ We can find
some $n,k\in T$ with $n\circ k\notin\phi(T)$ because of $T^{2}\nsubseteq$
$\phi(T).$ Then $(x\oplus n)\circ(y\oplus k)\subseteq T-\phi(T).$ Since $T$ is
$\phi$-prime hyperideal, then $(x\oplus n)\subseteq T$ or $(y\oplus
k)\subseteq T.$ So we have $x\in T$ or $y\in T.$ Hence $T$ is prime hyperideal
of $\Re.$
\end{proof}

\begin{corollary}
\label{corol1} If $T$ is a $\phi$-prime hyperideal of $\Re$ with $\phi
\leq\phi_{3}$, then $T$ is $w$-prime hyperideal.
\end{corollary}

\begin{proof}
We know that $T$ is $\phi$-prime hyperideal while $T$ is prime, for every
$\phi.$ Therefore $T$ is $w$-prime hyperideal of $\Re$. Assume that $T~$is not
prime. By \ref{thm5} $T^{2}\subseteq$ $\phi(T)\subseteq$ $T^{3}.$ So
$\phi(T)=T^{n}$ for every $n\geq2$ . Therefore $T$ is $n$-almost prime
hyperideal for every $n\geq2$. Hence $T$ is $w$-prime hyperideal of $\Re$.
\end{proof}

Some characterizations of $\phi$-prime hyperideals are provided.

\begin{theorem}
\label{Thmchar} Let $N$ be a proper hyperideal of the commutative
Krasner hyperring $\Re$ and let $\phi:L(\Re)\rightarrow L(\Re)\cup
\{\emptyset\}$ be a function. Then the following statements hold:

(i) $N$ is $\phi$-prime hyperideal of $\Re$.

(ii) For $a\in\Re-N$, $(N:a)=N\cup(\phi(N):a)$.

(iii) For $a\in\Re-N$, $(N:a)=N$ or $(N:a)=(\phi(N):a)$.

(iv) For each hyperideals $K$ and $L$ of $\Re$ such that $K\circ L\subseteq N$
and $K\circ L\nsubseteq\phi(N)$, we have $K\subseteq N$ or $L\subseteq N.$
\end{theorem}

\begin{proof}
$(i)\Longrightarrow(ii)$ Take $a\in\Re-N.$ Suppose that $b\in(N:a),$ then
$a\circ b\in N.$ If$\ a\circ b\notin\phi(N),$ then$\ b\in N$ since $N$ is
$\phi$-prime hyperideal of $\Re.$ If $a\circ b\in\phi(N)$, then $b\in
(\phi(N):a).$ Hence $(N:a)\subseteq N\cup(\phi(N):a)$. Other side holds since
the assumption of $\phi(N)\subseteq N.$

$(ii)\Longrightarrow(iii)$ It is obvious because of $(N:a)$ is a hyperideal of
$\Re$.

$(iii)\Longrightarrow(iv)$ Let $K$ and $L$ be hyperideals of $\Re$ such that
$K\circ L\subseteq N$ $-\phi(N).$ Assume that $K\nsubseteq N.$ Then there
exists an element $a\in K-N,$ by $(iii)$ we have $(N:a)=N$ or $(N:a)=(\phi
(N):a).$ If $a\circ L\nsubseteq\phi(N)$, then $L\nsubseteq(\phi(N):a).$ Since
$(iii)$ holds, then $L\subseteq(N:a)=N.$ We are done. Assume that $a\circ L\subseteq
\phi(N).$ Since $K\circ L\nsubseteq\phi(N),$ then we can choose an element $x\in K$ such
that $x\circ L\subseteq N-\phi(N).$ If $x\notin N,$ then by $(iii)$, $(N:x)=N$ or
$(N:x)=(\phi(N):x).$ Since $L\subseteq(N:x)$, but $L\nsubseteq(\phi(N):x)$, then we
conclude that $L\subseteq(N:x)=N.$ Assume that $x\in N$. We have
$a\oplus x\subseteq K-N,$ and also note that $(a\oplus x)\circ L\subseteq
N-\phi(N)$, since $a\circ L\subseteq\phi(N)$ and $x\circ L\nsubseteq\phi(N).$
Then $L\subseteq(N:a\oplus x)=N$ which completes the proof.

$(iv)\Longrightarrow(i)$ Let $a\circ b\in N-\phi(N).~$Then $(a)\circ
(b)\subseteq N,$ but $(a)\circ(b)\nsubseteq\phi(N).$ So $(a)\subseteq N$ or
$(b)\subseteq N.$ Then $a\in N$ or $b\in N.$
\end{proof}

In Theorem \ref{thm5} we show that if $T$ is a $\phi$-prime hyperideal and
$T$ is not prime, then $T^{2}\subseteq$ $\phi(T).$ 

\begin{corollary}
Let $T~$be a $\phi$-prime hyperideal that is not prime. Then $T\circ
\sqrt{\phi(T)}\subseteq$ $\phi(T)$.
\end{corollary}

\begin{proof}
Let $a\in\sqrt{\phi(T)}.$ If$ a\in T$, then $a\circ T\subseteq T^{2}\subseteq\phi(T)$
from the Theorem \ref{thm5}. Let us suppose that $a\notin T.$ From the main
Theorem \ref{Thmchar}, $(T:a)=T$ or $(T:a)=(\phi(T):a)$ as $T\subseteq(T:a)$
gives $a\circ T\subseteq\phi(T).$ Let us suppose $(T:a)=T.$ Assume that $a^{n}%
\in\phi(T)$ but $a^{n-1}\notin\phi(T).$ Then $a^{n}\in T$, so $a^{n-1}%
\in(T:a)=T.$ Hence $a^{n-1}\in T-\phi(T)$, so $a\in T,$ which is a
contradiction. That proof is similar to Corollary 14\cite{2}.
\end{proof}

In the following, we give a proposition about quotient and localization of a hyperring.
Let $S$ be a multiplicatively closed subset of a Krasner hyperring $\Re$%
. Define $\phi:L(\Re)\rightarrow L(\Re)\cup\{\emptyset\}$ and $\phi
_{S}:L(\Re_{S})\rightarrow L(\Re_{S})\cup\{\emptyset\}$ with $\phi
_{S}(M)=(\phi(M\cap\Re))_{S}.$ Also $\phi_{S}(M)=\emptyset$, where $\phi
(M\cap\Re)=\emptyset.~$

Let $N,M$ be hyperideals of $\Re$ and $N\subseteq M,$ define $\phi_{M}%
:L(\Re/M)\rightarrow L(\Re/M)\cup\{\emptyset\}$ with\ $\phi_{M}(N/M)=(\phi
(N)\oplus M)/M.$ Also $\phi_{M}(N/M)=\emptyset~$where $\phi(N)=\emptyset.$

\begin{proposition}
Suppose that $\phi:L(\Re)\rightarrow L(\Re)\cup\{\emptyset\}$ is a function
and $T~$is a$\ \phi$-prime hyperideal of $\Re$.

(i) If $M$ is a hyperideal of $\Re$ with $M\subseteq T,$ then $T/M$ is
$\phi_{M}$-prime hyperideal of $\Re/M.$

(ii) Let $S$ be a multiplicatively closed subset of $\Re$ with $T\cap
S=\emptyset$ and $\phi(T)_{S}\subseteq\phi_{S}(T_{S}).$ Then $T_{S}$ is a
$\phi_{S}$-prime hyperideal of $\Re_{S}.$
\end{proposition}

\begin{proof}
(i) Let $x,y\in\Re.$ Suppose that $(x\oplus M)\circ(y\oplus M)\subseteq
T/M-\phi_{M}(T/M).$ So $x\circ y\oplus M\subseteq T/M-(\phi(T\oplus M))/M.$
Then $x\circ y\in T-\phi(T\oplus M).~$We find that $x\circ y\in T-\phi(T)$ and
then $x\in T$ or $y\in T$. Therefore $x\oplus M\subseteq T/M$ or $y\oplus
M\subseteq T/M.$ So $T/M$ is $\phi_{M}$-prime hyperideal\ of $\Re/M$.

(ii) Let $\frac{a}{s}\circ\frac{b}{t}\in T_{S}-\phi_{S}(T_{S}),$ for some
$a,b\in\Re;\ s,t\in S.\ $Then we have $p\circ a\circ b\in T$ for some $p\in S$
but $q\circ a\circ b\notin\phi(T)\cap\Re$ for every $q\in S.$ Now if $q\circ
a\circ b\in\phi(T),$ then $\frac{a}{s}\circ\frac{b}{t}\in\phi(T)_{S}%
\subseteq\phi_{S}(T_{S})$ which is a contradiction. So $p\circ a\circ b\in
T-\phi(T)~$and since $T$ is $\phi$-prime hyperideal of $\Re$, then we get either $p\circ a\in T$ or $b\in T.$ Hence $\frac{a}{s}\in T_{S}$ or
$\frac{b}{t}\in T_{S}$, since $T\cap S=\emptyset.$ Thus $T_{S}$ is $\phi_{S}%
$-prime hyperideal of $\Re_{S}$.
\end{proof}

\begin{theorem}
(i) Let $X$ and $Y$ be commutative Krasner hyperrings and $N$ be a weakly prime
hyperideal of $X$. Then $M=N\otimes Y~$is a $\phi$-prime hyperideal of
$\Re=X\otimes Y$, for all $\phi$ with $\phi_{w}\leq\phi\leq\phi_{1}.$

(ii) Suppose $\Re$ is a commutative Krasner hyperring and $M$ is a finitely
generated proper hyperideal of $\Re$. Assuming $M\ $is a $\phi$-prime
hyperideal with $\phi\leq\phi_{3}.~$Then $M$ is either weakly prime or
$M^{2}\neq0$ is idempotent and $\Re$ decomposes as $X\otimes Y$ where
$Y=M^{2}$ and $M=N\otimes Y$ with $N$ be weakly prime. As a result $M$ is
$\phi$-prime for each $\phi$ with $\phi_{w}\leq\phi\leq\phi_{1}.$
\end{theorem}

\begin{proof}
(i) Let $N$ be a weakly prime hyperideal of $X$. Then $M=N\otimes Y~$not
necessarily to be a weakly prime hyperideal of $\Re=X\otimes Y;$ actually $M$
is weakly prime hyperideal if and only if $M~$is prime hyperideal.
Nevertheless $M~$is a $\phi$-prime hyperideal for each $\phi$ with $\phi
_{w}\leq\phi.$ If $N$ is prime, then $M$ is prime and hence $\phi$-prime for
all $\phi.$ Assume that $N$ is not prime. Then $N^{2}=0.$ Therefore
$M^{2}=0\otimes Y$ and hereby $\phi_{w}(M)=0\otimes Y$. Hence we can write
$M-\phi_{w}(M)$ in this form $M-\phi_{w}(M)=N\otimes Y-0\otimes
Y=(N-\{0\})\otimes Y.$ Let take $(a_{1},a_{2})\circ(b_{1},b_{2})=(a_{1}\circ
b_{1},a_{2}\circ b_{2})\in$ $M-\phi_{w}(M)$ . It means $a_{1}\circ b_{1}\in
N-\{0\},$ so $a_{1}\in N$ or $b_{1}\in N.$ Then $(a_{1},a_{2})\in$ $M$ or
$(b_{1},b_{2})\in$ $M.$ Hence $M~$is $\phi_{w}$-prime, therefore $\phi$-prime hyperideal.

(ii) If $M$ is prime hyperideal, then $M$ is weakly prime hyperideal of $\Re$.
Suppose that $M$ is not prime. From Theorem \ref{thm5}, $M^{2}\subseteq$
$\phi(M);$ and hence $M^{2}\subseteq$ $\phi(M)\subseteq\phi_{3}(M)=M^{3}.$ So
$M^{2}=M^{3},$ it means $M^{2}$ is idempotent. Since $M^{2}$ is finitely
generated, then $M^{2}=(m)$ for some idempotent $m\in\Re.$ Assume that $M^{2}=0$.
Then $\phi(M)\subseteq$ $M^{3}=0$. Hereby $\phi(M)=0.$ Consequently $M$ is
weakly prime hyperideal of $\Re$. Now suppose $M^{2}\neq0.$ Take that
$Y=M^{2}=\Re\circ m$ and $X=\Re\circ(1\ominus m),$ so $\Re$ decomposes as
$X\otimes Y$ where $Y=M^{2}.~$Let $I=M\circ(1\ominus m),$ so $M=N\otimes Y$
where $N^{2}=(M\circ(1\ominus m))^{2}=M^{2}\circ(1\ominus m)^{2}%
=(m)\circ(1\ominus m)=0.$ To show $N$ is weakly prime hyperideal, let $a\circ
b\in N^{2}-\{0\};$ so $(a,1)\circ(b,1)=(a\circ b,1)\in$ $N\otimes Y-(N\otimes
Y)^{2}=N\otimes Y-0\otimes Y\subseteq M-\phi(M)$. Since $\phi\leq\phi_{3}$, then it follows $\phi(M)\subseteq$ $M^{3}=(N\otimes Y)^{3}=0\otimes Y$. We obtain that
$(a,1)\in$ $M$ or $(b,1)\in$ $M.$ Therefore $a\in$ $N$ or $b\in$ $N$. As a
consequence~$N$ is weakly prime hyperideal of $\Re$.
\end{proof}

\begin{proposition}
Let $\Re_{1}$ and $\Re_{2}$ be commutative Krasner hyperrings and let
$\varphi_{i}:L(\Re_{i})\rightarrow L(\Re_{i})\cup\{\emptyset\}$ be\ a function
for $i=1,2.$ Take $\phi=\varphi_{1}\times\varphi_{2}$ and $\Re=\Re_{1}%
\times\Re_{2}.$ Then $N$ is $\phi$-prime hyperideal of $\Re$ if and only if
$N$ is one of the following types:

(i) $N=N_{1}\times N_{2},\ $where $N_{i}\ $is a proper hyperideal of $\Re
_{i}\ $with $\varphi_{i}(N_{i})=N_{i}.$

(ii) $N=N_{1}\times\Re_{2},\ $where $N_{1}\ $is $\varphi_{1}$-prime hyperideal
of $\Re_{1}$ that should be prime if $\varphi_{2}(\Re_{2})\neq\Re_{2}.$

(iii)\ $N=\Re_{1}\times N_{2},\ $where $N_{2}\ $is $\varphi_{2}$-prime
hyperideal of $\Re_{2}$ that should be prime if $\varphi_{1}(\Re_{1})\neq
\Re_{1}.$
\end{proposition}

\begin{proof}
($\Longrightarrow$) (i) Obviously $N$ is $\phi$-prime hyperideal, since
$N_{1}\times N_{2}-\phi(N_{1}\times N_{2})=\emptyset$

(ii) Let $N_{1}\ $is $\varphi_{1}$-prime hyperideal of $\Re_{1}$ and
$\varphi_{2}(\Re_{2})\neq\Re_{2}$. Assume that $(a_{1}\circ b_{1},a_{2}\circ
b_{2})=(a_{1},a_{2})\circ(b_{1},b_{2})\in N_{1}\times\Re_{2}-\varphi_{1}%
(N_{1})\times\varphi_{2}(\Re_{2})=(N_{1}-\varphi_{1}(N_{1}))\times(\Re
_{2}-\varphi_{2}(\Re_{2}))$ for $(a_{1},a_{2}),(b_{1},b_{2})\in\Re=\Re
_{1}\times\Re_{2}.~$Thus $a_{1}\circ b_{1}\in N_{1}-\varphi_{1}(N_{1})$ and then
either $a_{1}\in N_{1}$ or $b_{1}\in N_{1}$. So $(a_{1},a_{2})\in N_{1}\times\Re_{2}$ or $(b_{1}%
,b_{2})\in N_{1}\times\Re_{2}.$ Therefore $N_{1}\times\Re_{2}$ is $\phi$-prime
hyperideal of $\Re$.

(iii) The proof is similar to (ii).

($\Longleftarrow$) Assume that $N$ is a $\phi$-prime hyperideal of $\Re$ where
$\varphi_{i}(N_{i})\neq N_{i}.$ Let $a\circ b\in N_{1}-\varphi_{1}(N_{1})$ for
some $a,b\in\Re_{1}.$ So $(a,0)\circ(b,0)=(a\circ b,0)\in N-\phi(N).$ Since
$N$ is a $\phi$-prime hyperideal of $\Re,$ then $(a,0)\in N$ or $(b,0)\in N$.
So $a\in N_{1}$ or $b\in N_{1}.$ Therefore $N_{1}\ $is $\varphi_{1}$-prime
hyperideal of $\Re_{1}.~$Similarly we can find $N_{2}\ $is $\varphi_{2}$-prime
hyperideal of $\Re_{2}.$ Now we need to show that $N_{1}=\Re_{1}$ or
$N_{2}=\Re_{2}.$ Suppose that $N_{2}\not =\Re_{2}.$ Take $b_{1}\in
N_{1}-\varphi_{1}(N_{1}),$ $b_{2}\in\Re_{2}-N_{2}.$ Then note that
$(1,0)\circ(b_{1},b_{2})=(b_{1},0)\in N-\phi(N).$ Since $N_{2}\not =\Re_{2},$ then
$(1,0)\in N$ and so we find $1\in N_{1}.$ We get $N_{1}=\Re_{1}.$ Similarly
one can easily find $N_{2}=\Re_{2}$, if $N_{1}\not =\Re_{1}$. Without loss of
generality $N_{1}\not =\Re_{1}$. Now let we show that $N_{1}$ is prime hyperideal
with $\varphi_{2}(\Re_{2})\neq\Re_{2}$. For $m^{\prime}\in\Re_{2}-\varphi
_{2}(\Re_{2})$, let $x\circ m\in N_{1}$ for some $x,m\in\Re_{1}$. Then we
conclude that $(x,1)\circ(m,m^{\prime})=(x\circ m,m^{\prime})\in N-\phi(N)$.
Since $N$ is a $\phi$-prime hyperideal of $\Re$, then we find $(x,1)\in N$ or
$(m,m^{\prime})\in N$ which implies that $x\in N_{1}$ or $m\in N_{1}$.
Therefore $N_{1}$ is a prime hyperideal of $\Re_{1}.$ If $\varphi_{1}(\Re
_{1})\neq\Re_{1}$ and $N_{1}=\Re_{1}$, then similarly one can prove that
$N_{2}$ is a prime hyperideal of $\Re_{2}$.
\end{proof}

\section{Generalizations of Primary Hyperideals in Krasner Hyperrings}

Similar to the previous section, we consider $(\Re,\oplus,\circ)$ to be a
commutative Krasner hyperring with nonzero unit. We denote the set of all
hyperideals of $\Re$ by $L(\Re)$. Let define $\phi$-primary hyperideal.

\begin{definition}
Let $\Re$ be a commutative hyperring and $N$ be a proper hyperideal of$\ \Re.$
Let $\phi$ be a function such that $\phi:L(\Re)\rightarrow L(\Re
)\cup\{\emptyset\}$. $N$ is called $\phi$-primary hyperideal of $\Re$ if
$a\circ b\in N-$ $\phi(N),$ then either $a\in N$ or $b^{k}\in N$ for some
$a,b\in\Re,~k\in%
\mathbb{N}
$.
\end{definition}

\begin{example}
\label{exam2} Let $\Re$ be a commutative Krasner hyperring. Then we define
functions $\phi:L(\Re)\rightarrow L(\Re)\cup\{\emptyset\}$ such as in Example
\ref{exam1}. Also we have the same order, $\phi_{\oslash}\leq\phi_{0}\leq
\phi_{w}\leq...\leq\phi_{n+1}\leq\phi_{n}\leq\phi_{n-1}\leq...\leq\phi_{2}%
\leq\phi_{1}.$
\end{example}

\begin{example}
\label{ex1primary} Let $\Re$ be a hyperring and $N $be a proper hyperideal of
$\Re.$

(i) $N$ is primary hyperideal if and only if it is $\phi_{\emptyset}%
$-primary hyperideal.

(ii) $N $ is weakly primary hyperideal if and only if it is $\phi_{0}%
$-primary hyperideal.

(iii) $N$ is almost primary hyperideal if and only if it is $\phi_{2}%
$-primary hyperideal.

(iv) $N$ is $n$-almost primary hyperideal if and only if it is $\phi_{n}%
$-primary hyperideal.

(v) $N$ is $w$-primary hyperideal if and only if it is $\phi_{w}$-primary hyperideal.
\end{example}

\begin{proposition}
Let $N~$be a proper hyperideal of $\Re.$

(1) If $\sigma_{1},$ $\sigma_{2}$ are two functions with $\sigma_{1}\leq$
$\sigma_{2}$ such that $\sigma_{1},$ $\sigma_{2}:L(\Re)\rightarrow L(\Re
)\cup\{\emptyset\}$ and $N$ is $\sigma_{1}$-primary, then $N$ is $\sigma_{2}$-primary.

(2) (i) $N$ is primary hyperideal $\Longrightarrow$ $N$ is weakly primary
hyperideal $\Longrightarrow$ $N$ is $w$-primary hyperideal $\Longrightarrow$
$N$ is $(n+1)$-almost primary hyperideal $\Longrightarrow$ $N$ is $n$-almost
primary hyperideal, for $n\geq2\Longrightarrow$ $N$ is almost primary hyperideal.

\ \ \ (ii) $N$ is $w$-primary hyperideal if and only if $N$ is $n$-almost
primary hyperideal for all $n\geq2.$
\end{proposition}

\begin{proof}
(1) We suppose $N$ is a $\sigma_{1}$-primary hyperideal of $\Re$. Take $a\circ
b\in N-\sigma_{2}(N)$ for $a,b\in\Re.$ It means $a\circ b\in N-\sigma_{1}(N)$.
Because of $N$ is $\sigma_{1}$-primary hyperideal of $\Re$, $a\in N$ or
$b^{k}\in N$, for some $k\in%
\mathbb{N}
$. Therefore $N$ is $\sigma_{2}$-primary hyperideal.

(2) (i) We can obtain it by the ordering of the $\phi_{i}^{\prime}$s given in
\ref{exam2}. (ii) Assume that $a\circ b\in N-N^{n}$ for all $n\geq2$. Then
$a\circ b\in N-\cap_{n=2}^{\infty}N^{n}$ and we find $a\circ b\in N-\cap
_{n=1}^{\infty}N^{n}.$ Since $N$ is $w$-primary hyperideal of $\Re,$ then we have
$a\in N$ or $b^{k}\in N$, for some $k\in%
\mathbb{N}
$ . Conversely, if $a\circ b\in N-\cap_{n=1}^{\infty}N^{n},$ then $a\circ b\in
N-$ $N^{n}$ for some $n\geq1.$ Actually $a\circ b\in N-$ $N^{n}$ for some
$n\geq2$. Therefore $a\in N$ or $b^{k}\in N$ for some $k\in%
\mathbb{N}
,$ since $N$ is $n$-almost primary for all $n\geq2.$
\end{proof}

The next theorem shows us how to determine $\phi$-primary hyperideal to be
primary. We call it as a characterization.

\begin{theorem}
\label{thm4.3} Let $\phi:L(\Re)\rightarrow L(\Re)\cup\{\emptyset\}$ be a
function and $T~$be a proper hyperideal of $\Re$, such that $T~$ is a $\phi$-primary
hyperideal of $\Re.$ If $T$ is not primary, then $T^{2}\subseteq$ $\phi(T).$
Hence at the same time it means if $T^{2}\nsubseteq$ $\phi(T)$, then $T$ is
primary hyperideal of $\Re.$
\end{theorem}

\begin{proof}
Assume that$~T^{2}\nsubseteq$ $\phi(T).~$We need to see that $T$ is primary.
Take some $x\circ y\in T$ for $x,y\in\Re.$ If $x\circ y\notin\phi(T)$, then since
$T$ is $\phi$-primary, $x\in T$ or $y^{k}\in T$ for some $k\in%
\mathbb{N}
.$ If $x\circ y\in\phi(T),$ assuming that $x\circ T\nsubseteq$ $\phi(T),$ we get
$x\circ p_{0}\notin\phi(T)$, where $p_{0}\in T.$ We have $x\circ(y\oplus
p_{0})\subseteq T-\phi(T)$. Thus $x\in T$ or $(y\oplus p_{0})^{k}\subseteq T$
for some $k\in%
\mathbb{N}
$. It follows $x\in T$ or $y^{k}\in T.$ Now we can assume that $x\circ
T\subseteq$ $\phi(T).$ (In the same way we can assume that $y\circ T\subseteq$
$\phi(T)$). Because of $T^{2}\nsubseteq$ $\phi(T),$ there exist $p_{1}%
,q_{1}\in T$ with $p_{1}\circ q_{1}\notin\phi(T).$ Then $(x\oplus p_{1}%
)\circ(y\oplus q_{1})\subseteq T-\phi(T).$ As $T$ is $\phi$-primary, so
$(x\oplus p_{1})\subseteq T$ or $(y\oplus q_{1})^{m}\subseteq T$ for some
$m\in%
\mathbb{N}
.$ Consequently $T$ is primary hyperideal of $\Re$.
\end{proof}

\begin{corollary}
If $T~$is a $\phi$-primary hyperideal of $\Re$, where $\phi\leq\phi_{3},$ then
$T$ is $w$-primary hyperideal.
\end{corollary}

\begin{proof}
That proof is similar to Corollary \ref{corol1}.
\end{proof}

In the following, some characterizations of $\phi$-primary hyperideals are provided.

\begin{theorem}
\label{Thmchar2} Let $N$ be a proper hyperideal of the commutative
Krasner hyperring $\Re$ and let $\phi:L(\Re)\rightarrow L(\Re)\cup
\{\emptyset\}$ be a function. The following statements hold:

$(i)$ $N$ is $\phi$-primary hyperideal of $\Re$.

$(ii)$ For $a\in\Re-\sqrt{N};$ $(N:a)=N\cup(\phi(N):a)$.

$(iii)$ For $a\in\Re-\sqrt{N};$ $(N:a)=N$ or $(N:a)=(\phi(N):a)$.

$(iv)$ For each hyperideals $K$ and $L$ of $\Re$, if $K\circ L\subseteq N$ and
$K\circ L\nsubseteq\phi(N),$ then $K\subseteq N$ or $L\subseteq\sqrt{N}.$
\end{theorem}

\begin{proof}
$(i)\Longrightarrow(ii)$ Let $N$ is $\phi$-primary hyperideal of $\Re$.~It is
obvious that $N\cup(\phi(N):a)\subseteq(N:a).$ To prove the other side; for
every $b\in(N:a),$ we have $a\circ b\in N.$ If$\ a\circ b\in N-\phi(N),$
then$\ b\in N$ since $N$ is $\phi$-primary and $a\in\Re-\sqrt{N}.$ If $a\circ
b\in\phi(N)$ then $b\in(\phi(N):a).$ So $(N:a)\subseteq N\cup(\phi(N):a)$.
Hence $(N:a)=N\cup(\phi(N):a)$.

$(ii)\Longrightarrow(iii)$ It is obvious because of $(N:a)$ is a hyperideal of
$\Re$.

$(iii)\Longrightarrow(iv)$ Let $K$ and $L$ be hyperideals of $\Re$, with
$K\circ L\subseteq N-\phi(N).$ Assume that $L\nsubseteq\sqrt{N}.$ Then there
exists an element $a\in L-\sqrt{N}~$by $(iii)$ we have $(N:a)=N$ or
$(N:a)=(\phi(N):a).$ If $K\circ a\nsubseteq\phi(N)$, then $K\nsubseteq
(\phi(N):a).$ Since $(iii)$ holds, then $K\subseteq(N:a)=N.$ We are done. Assume that
$K\circ a\subseteq\phi(N).$ Since $K\circ L\nsubseteq\phi(N),$ then we can choose an
element $x\in L$ such that $K\circ x\subseteq N-\phi(N).$ If $x\notin\sqrt
{N},$ then by $(iii)$, $(N:x)=N$ or $(N:x)=(\phi(N):x).$ Since $K\subseteq(N:x)$, but
$K\nsubseteq(\phi(N):x)$, then we conclude that $K\subseteq(N:x)=N.$ So assume that
$x\in\sqrt{N}.$ Then we have $a\oplus x\subseteq L-\sqrt{N},$ and also note
that $K\circ(a\oplus x)\subseteq N-\phi(N)$, since $K\circ a\subseteq\phi(N)$
and $K\circ x\nsubseteq\phi(N).$ Then $K\subseteq(N:a\oplus x)=N$ which
completes the proof.

$(iv)\Longrightarrow(i)$ Let $a\circ b\in N-\phi(N).$ Then $(a)\circ
(b)\subseteq N,$ but $(a)\circ(b)\nsubseteq\phi(N).$ So $(a)\subseteq N$ or
$(b)\subseteq\sqrt{N}.$ Therefore $a\in N$ or $b^{m}\in N$ for some $m\in%
\mathbb{N}
.$
\end{proof}

In Theorem \ref{thm4.3}, we prove that if $T$ is a $\phi$-primary hyperideal
that is not primary, then $T^{2}\subseteq$ $\phi(T).$

\begin{proposition}
Let $\phi:L(\Re)\rightarrow L(\Re)\cup\{\emptyset\}$ be a function
and $T~$ be a $\phi$-primary hyperideal of $\Re$.

(i) If $T$ is an hyperideal of $\Re$ with $M\subseteq T,$ then $T/M$ is
$\phi_{M}$-primary hyperideal of $\Re/M.$

(ii) Let $S$ be a multiplicatively closed subset of $\Re$ with $T\cap
S=\emptyset~$and $\phi(T)_{S}\subseteq\phi_{S}(T_{S}).$ Then $T_{S}$ is an
$\phi_{S}$-primary hyperideal of $\Re_{S}.$
\end{proposition}

\begin{proof}
(i) Let $x,y\in\Re.$ We assume that $(x\oplus M)\circ(y\oplus M)\subseteq
T/M-\phi_{M}(T/M),$ it means $x\circ y\oplus M\subseteq T/M-(\phi(T\oplus
M))/M.$ We find $x\circ y\in T-\phi(T\oplus M).$ Then $x\circ y\in T-\phi(T)$
and so $x\in T$ or $y^{k}\in T$ for some $k\in%
\mathbb{N}
.$ Hence $x\oplus M\subseteq T/M$ or $(y\oplus M)^{k}\subseteq T/M.$ This
gives us $T/M$ is $\phi_{M}$-primary hyperideal of $\Re/M$.

(ii) Let $\frac{a}{s}\circ\frac{b}{t}\in T_{S}-\phi_{S}(T_{S}),$ for some
$a,b\in\Re; s,t\in S.$ Then we have $p\circ a\circ b\in T$ for some $p\in S$
but $q\circ a\circ b\notin\phi(T)\cap\Re$ for every $q\in S.$ Now if $q\circ
a\circ b\in\phi(T),$ then $\frac{a}{s}\circ\frac{b}{t}\in\phi(T)_{S}%
\subseteq\phi_{S}(T_{S})$ gives us a contradiction. So $p\circ a\circ b\in
T-\phi(T)$ and since $T$ is $\phi$-primary, then we get either
$p\circ a\in T$ or $b^{k}\in T$ for some $k\in%
\mathbb{N}
.$ Therefore $\frac{a}{s}\in T_{S}$ or $\frac{b^{k}}{t^{k}}\in T_{S}$ because
of $\sqrt{T}\cap S=\emptyset~($it is same with $T\cap S=\emptyset).$ Thus
$T_{S}$ is $\phi_{S}$-primary hyperideal of $\Re_{S}$.
\end{proof}

\begin{theorem}
(i) Let $X, Y$ be commutative Krasner hyperrings and $N$ be a weakly primary
hyperideal of $X$. Then $M=N\otimes Y~$ is a $\phi$-primary hyperideal of
$\Re=X\otimes Y$, for all $\phi$ with $\phi_{w}\leq\phi\leq\phi_{1}.$

(ii) Let $\Re$ be a commutative Krasner hyperring and $M$ be a finitely
generated proper hyperideal of $\Re$, such that $M$ is a $\phi$-primary
hyperideal with $\phi\leq\phi_{3}.$ Then $M$ is either weakly primary or
$M^{2}\neq0$ is idempotent and $\Re$ decomposes as $X\otimes Y$, where
$Y=M^{2}$ and $M=N\otimes Y$, where $N$ is weakly primary. As a result $M$ is
$\phi$-primary for each $\phi$ with $\phi_{w}\leq\phi\leq\phi_{1}.$
\end{theorem}

\begin{proof}
(i) Let $N$ be a weakly primary hyperideal of $X$. $M$ is weakly primary if
and only if $M$ is primary. Nevertheless $M~$is a $\phi$-primary hyperideal
for each $\phi$ with $\phi_{w}\leq\phi.$ If $N$ is primary, then $M$ is
primary and hence $\phi$-primary for all $\phi.$ Assume that $N$ is not
primary. Then $N^{2}=0.$ Therefore $M^{2}=0\otimes Y$ and hereby $\phi
_{w}(M)=0\otimes Y$. Hence we can write $M-\phi_{w}(M)$ in this form
$M-\phi_{w}(M)=N\otimes Y-0\otimes Y=(N-\{0\})\otimes Y.$ Let take
$(a_{1},a_{2})\circ(b_{1},b_{2})=(a_{1}\circ b_{1},a_{2}\circ b_{2})\in$
$M-\phi_{w}(M)$. It means $a_{1}\circ b_{1}\in N-\{0\},$ so $a_{1}\in N$ or
$b_{1}^{s}\in N$ for some $s\in%
\mathbb{N}
$. Then $(a_{1},a_{2})\in$ $N\otimes Y$ or $(b_{1},b_{2})^{s}\in$ $N\otimes
Y$. Hence $M=N\otimes Y~$is $\phi_{w}$-primary, therefore $M$ is $\phi
$-primary hyperideal.

(ii) If $M$ is primary hyperideal, then $M$ is weakly primary hyperideal of
$\Re$. Suppose that $M$ is not primary. From Theorem \ref{thm4.3},
$M^{2}\subseteq$ $\phi(M)$ and hence $M^{2}\subseteq$ $\phi(M)\subseteq
\phi_{3}(M)=M^{3}.$ So $M^{2}=M^{3},$ which means $M^{2}$ is idempotent. Since
$M^{2}$ is finitely generated, then $M^{2}=(m)$ for some idempotent $m\in\Re.$
Assume that $M^{2}=0$. Then $\phi(M)\subseteq$ $M^{3}=0$. Hereby $\phi(M)=0.$
Consequently, $M$ is weakly primary hyperideal of $\Re$. Now let us suppose $M^{2}%
\neq0.$ Take $Y=M^{2}=\Re\circ m$ and $X=\Re\circ(1\ominus m),$ so $\Re$
decomposes as $X\otimes Y$, where $Y=M^{2}.$ Let $N=M\circ(1\ominus m),$ so
$M=N\otimes Y$, where $N^{2}=(M\circ(1\ominus m))^{2}=M^{2}\circ(1\ominus
m)^{2}=(m)\circ(1\ominus m)=0.$ To show $N$ is weakly primary hyperideal, let
$a\circ b\in N^{2}-\{0\}$. Thus $(a,1)\circ(b,1)=(a\circ b,1)\in$ $N\otimes
Y-(N\otimes Y)^{2}=N\otimes Y-0\otimes Y\subseteq M-\phi(M)$, since $\phi
\leq\phi_{3}$ implies $\phi(M)\subseteq$ $M^{3}=(N\otimes Y)^{3}=0\otimes Y$.
We obtain that $(a,1)\in$ $M$ or $(b,1)^{t}\in$ $M$, for some $t\in%
\mathbb{N}
.$ Therefore $a\in$ $N$ or $b^{t}\in$ $N$. As a consequence, $N$ is weakly
primary hyperideal of $\Re$.
\end{proof}

\begin{proposition}
Let $\Re_{1}$ and $\Re_{2}$ be commutative Krasner hyperrings and let
$\varphi_{i}:L(\Re_{i})\rightarrow L(\Re_{i})\cup\{\emptyset\}$ be a function,
for $i=1,2.$ Take $\phi=\varphi_{1}\times\varphi_{2}$ and $\Re=\Re_{1}%
\times\Re_{2}.$ Then $N$ is $\phi$-primary hyperideal of $\Re$ if and only if
$N$ is one of the following types:

(i) $N=N_{1}\times N_{2},\ $where $N_{i}\ $is a proper hyperideal of $\Re
_{i}\ $with $\varphi_{i}(N_{i})=N_{i}.$

(ii) $N=N_{1}\times\Re_{2},\ $where $N_{1}\ $is $\varphi_{1}$-primary
hyperideal of $\Re_{1}$ that should be primary if $\varphi_{2}(\Re_{2})\neq
\Re_{2}.$

(iii)\ $N=\Re_{1}\times N_{2},\ $where $N_{2}\ $is $\varphi_{2}$-primary
hyperideal of $\Re_{2}$ that should be primary if $\varphi_{1}(\Re_{1})\neq
\Re_{1}.$
\end{proposition}

\begin{proof}
($\Longrightarrow$) (i) Obviously $N$ is $\phi$-primary hyperideal, since
$N_{1}\times N_{2}-\phi(N_{1}\times N_{2})=\emptyset.$

(ii) Let $N_{1}\ $ be $\varphi_{1}$-primary hyperideal of $\Re_{1}$ and
$\varphi_{2}(\Re_{2})\neq\Re_{2}$. Assume that $(a_{1}\circ b_{1},a_{2}\circ
b_{2})=(a_{1},a_{2})\circ(b_{1},b_{2})\in N_{1}\times\Re_{2}-\varphi_{1}%
(N_{1})\times\varphi_{2}(\Re_{2})=(N_{1}-\varphi_{1}(N_{1}))\times(\Re
_{2}-\varphi_{2}(\Re_{2}))$, for $(a_{1},a_{2}),(b_{1},b_{2})\in\Re=\Re
_{1}\times\Re_{2}.$ We have $a_{1}\circ b_{1}\in N_{1}-\varphi_{1}(N_{1}),$ then
either $a_{1}\in N_{1}$ or $b_{1}^{k}\in N_{1}$, for some $k\in%
\mathbb{N}
$, since $N_{1}\ $is $\varphi_{1}$-primary hyperideal of $\Re_{1}$. So
$(a_{1},a_{2})\in N_{1}\times\Re_{2}$ or $(b_{1}^{k},b_{2}^{k})=(b_{1}%
,b_{2})^{k}\in N_{1}\times\Re_{2}.$ Therefore $N_{1}\times\Re_{2}$ is $\phi
$-primary hyperideal of $\Re$.

(iii) The proof is similar to (ii).

($\Longleftarrow$) Suppose that $N$ is a $\phi$-primary hyperideal of $\Re$,
where $\varphi_{i}(N_{i})\neq N_{i}.$ Let $a\circ b\in N_{1}-\varphi_{1}%
(N_{1})$, for some $a,b\in\Re_{1}.$ Thus $(a,0)\circ(b,0)=(a\circ b,0)\in
N-\phi(N).$ Since $N$ is a $\phi$-primary hyperideal of $\Re,$ then $(a,0)\in
N$ or $(b,0)^{k}\in N$ for some $k\in%
\mathbb{N}
.$ So $a\in N_{1}$ or $b^{k}\in N_{1}.$ Therefore $N_{1}\ $is $\varphi_{1}%
$-primary hyperideal of $\Re_{1}.$ Similarly we can find $N_{2}\ $is
$\varphi_{2}$-primary hyperideal of $\Re_{2}.$ We have to show now that
$N_{1}=\Re_{1}$ or $N_{2}=\Re_{2}.$ Suppose that $N_{2}\not =\Re_{2}.$ Let
take $b_{1}\in N_{1}-\varphi_{1}(N_{1}),$ $b_{2}\in\Re_{2}-N_{2}.$ Then note
that $(1,0)\circ(b_{1},b_{2})=(b_{1},0)\in N-\phi(N).$ Since $N_{2}\not =%
\Re_{2},$ $(1,0)^{t}=(1,0)\in N$ ($\exists t\in%
\mathbb{N}
$), then we find $1\in N_{1}$ and so we get $N_{1}=\Re_{1}.$ Similarly one can
easily find $N_{2}=\Re_{2}$, if $N_{1}\not =\Re_{1}$. Without loss of
generality $N_{1}\not =\Re_{1}$. Now let we show that $N_{1}$ is primary
hyperideal with $\varphi_{2}(\Re_{2})\neq\Re_{2}$. For $m^{\prime}\in\Re
_{2}-\varphi_{2}(\Re_{2})$, let $x\circ m\in N_{1}$, for some $x,m\in\Re_{1}$.
Then we conclude that $(x,1)\circ(m,m^{\prime})=(x\circ m,m^{\prime})\in
N-\phi(N)$. Since $N$ is a $\phi$-primary hyperideal of $\Re$, then we find
$(x,1)^{s}\in N$ for some $s\in%
\mathbb{N}
$ or $(m,m^{\prime})\in N$ which implies that $x^{s}\in N_{1}$ or $m\in N_{1}%
$. Therefore $N_{1}$ is a primary hyperideal of $\Re_{1}.$ If $\varphi_{1}%
(\Re_{1})\neq\Re_{1}$ and $N_{1}=\Re_{1}$, then similarly one can prove that
$N_{2}$ is a primary hyperideal of $\Re_{2}$.
\end{proof}

\section{$\phi$-$\delta$-primary {Hyperideals in Krasner Hyperrings }}

Let $N$ be a proper hyperideal of hyperring $\Re$. Denote the set of all hyperideals of $\Re,$by $L(\Re)$ and denote the
set of all proper hyperideals of $\Re,$ by $L^{\ast}(\Re).$ The function $\phi
:L(\Re)\rightarrow L(\Re)\cup\{\emptyset\}$ is said to be reduction
function if $\phi(N)\subseteq N\ $and $N\subseteq M\ $implies $\phi
(N)\subseteq\phi(M)$, for each $N, M\in L(\Re)$ and $\delta$ be an expansion
function such that $\delta:L(\Re)\rightarrow L(\Re).$ Now, we give some
examples related to reduction and expansion functions.

\begin{example}
Let $\Re$ be a commutative Krasner hyperring with a nonzero identity. Let us consider the following functions $\delta$ on $L(\Re),$ for any $N\in
L(\Re),$

(i)$\ \delta_{0}(N)=N,\ $i.e., $\delta$ is the identity function.

(ii) $\delta_{1}(N)=\sqrt{N},\ $i.e., $\delta$ is the radical operation.

(iii)\ $\delta_{res}(N)=(N:M)\ $for a fixed $M\in L(\Re).$

$\ $(iv)$\ \delta_{ann}(N)=ann(ann(N)).\ $

(vi)\ $\delta_{M}(N)=N\oplus M\ $for a fixed $M\in L(\Re).$
\end{example}

All the above functions are examples of expansion on $L(\Re).$

\begin{example}
$~$Let $\Re\ $be a commutative Krasner hyperring with a nonzero identity.
Consider the following functions $\phi:L(\Re)\rightarrow L(\Re)\cup
\{\emptyset\}$ defined as follows: for any $N\in L(\Re):$

(i)\ $\phi_{\emptyset}(N)=\emptyset.$

(ii)\ $\phi_{0}(N)=0.$

(iii)\ $\phi_{1}(N)=N.$

(iv)\ $\phi_{2}(N)=N^{2}.$

(v)\ $\phi_{k}(N)=N^{k}.$

(vi)\ $\phi_{w}(N)=%
{\displaystyle\bigcap\limits_{i=1}^{\infty}}
N^{i}.$
\end{example}

All the above functions are reduction on $L(\Re).$ Remember that $\phi_{\oslash
}\leq\phi_{0}\leq\phi_{w}\leq...\leq\phi_{n+1}\leq\phi_{n}\leq\phi_{n-1}%
\leq...\leq\phi_{2}\leq\phi_{1}.$

\begin{definition}
Let $\delta$ be a hyperideal expansion, $\phi$ be a hyperideal reduction and
$N$ be a proper hyperideal of $\Re.$ $N$ is said to be a $\phi$-$\delta
$-primary hyperideal if $a\circ b\in N-$ $\phi(N),$ then either $a\in N$ or
$b\in\delta(N)$ for each $a,b\in\Re.$
\end{definition}

\begin{remark}
\cite{18} If $\delta_{1},\delta_{2},...,\delta_{n}$ are hyperideal expansions,
then $\delta=\cap_{i=1}^{n}\delta_{i}$ is also an hyperideal expansion.
\end{remark}

\begin{example}
\label{ex1} Let $\Re$ be a hyperring and $N$ be a proper hyperideal of
$\Re.$

(i) $N$ is prime hyperideal if and only if it is $\phi_{\emptyset}$%
-$\delta_{0}$-primary hyperideal \cite{18}.

(ii)\ $N\ $is primary hyperideal if and only if it is $\phi_{\emptyset}%
$-$\delta_{1}$-primary hyperideal \cite{18}.

(iii) $N$ is $\phi$-prime hyperideal if and only if it is $\phi$%
-$\delta_{0}$-primary hyperideal.

(iv) $N$ is $\phi$-primary hyperideal if and only if it is $\phi$-$\delta
_{1}$-primary.

(v)\ $N$ is $\delta$-primary hyperideal if and only if it is $\phi
_{\emptyset}$-$\delta$-primary hyperideal.

(vi) $N\ $is weakly prime hyperideal if and only if it is $\phi_{0}$%
-$\delta_{0}$-primary hyperideal.

(vii)\ $N\ $is almost prime hyperideal if and only if it is $\phi_{2}$%
-$\delta_{0}$-primary hyperideal.
\end{example}

Suppose that $\phi,\sigma$ are reductions on $L(\Re)$. Then we write
$\phi\leq\sigma$ if $\phi(N)\subseteq\sigma(N)$, for all $N\in L(\Re
).$ Similarly, for any two expansions $\delta,\gamma$ on $L(\Re); \delta
\leq\gamma$ if $\delta(N)\subseteq\gamma(N)$, for each $N\in L(\Re).$

\begin{definition}
Let $\Re$ be a hyperring and $N$ be a proper hyperideal of $\Re.$

(i)\ If $N$ is $\phi_{0}$-$\delta$-primary hyperideal, then $N$ is said to
be a weakly $\delta$-primary hyperideal of $\Re.$

(ii)\ If $N$ is $\phi_{2}$-$\delta$-primary hyperideal, then $N$ is said to
be an almost $\delta$-primary hyperideal of $\Re$.

(iii)\ If $N$ is $\phi_{n}$-$\delta$-primary hyperideal, then $N$ is said
to be an $n$-almost $\delta$-primary hyperideal of $\Re$.

(iv)\ If $N$ is $\phi_{w}$-$\delta$-primary hyperideal, then $N$ is said to
be a $w$-$\delta$-primary hyperideal of $\Re$.
\end{definition}

\begin{proposition}
\label{p1} Let $\Re$ be a hyperring, $N$ be a proper hyperideal of $\Re
$, $\phi,\varphi$ be reductions on $L(\Re)$ and $\delta,\gamma$ be
expansions on $L(\Re)$. The following statements hold:

(i) If $\phi\leq\sigma,$ then every $\phi$-$\delta$-primary hyperideal is a
$\sigma$-$\delta$-primary hyperideal.

(ii)\ If $\delta\leq\gamma,$ then every $\phi$-$\delta$-primary hyperideal
is a $\phi$-$\gamma$-primary hyperideal.

(iii) Every $\phi$-prime hyperideal is a $\phi$-$\delta$-primary hyperideal.

(iv) Every $\delta$-primary hyperideal is a $\phi$-$\delta$-primary hyperideal.

(v)\ Suppose that $N$ is proper hyperideal of $\Re$. $N$ is weakly $\delta
$-primary $\Rightarrow$ $N$ is $w$-$\delta$-primary hyperideal $\Rightarrow
N$ is $n$-almost $\delta$-primary, for each $n\geq2 \Rightarrow N$ is
almost $\delta$-primary.
\end{proposition}

\begin{proof}
$(i) (ii):$ Straightforward.

$(iii):$ It follows from (ii) and Example \ref{ex1} (iii), since $\delta_{0}%
\leq\delta.$

$(iv):$ It follows from (i) and Example \ref{ex1} (iii), since $\phi_{\emptyset
}\leq\phi.$

$(v):$ It follows from (i) and the fact that $\phi_{0}\leq\phi_{w}\leq\phi
_{n}\leq\phi_{2}.$
\end{proof}

\begin{proposition}
Let $\phi$ be a hyperideal reduction, $\delta$ be a hyperideal expansion and
$\{M_{i}:i\in\Delta\}$ be a directed family of $\phi$-$\delta$-primary
hyperideals of $\Re.$ Then $M=%
{\displaystyle\bigcup\limits_{i\in\Delta}}
M_{i} $is a $\phi$-$\delta$-primary hyperideal.
\end{proposition}

\begin{proof}
Let $\{M_{i}:i\in\Delta\}$ be a directed family of $\phi$-$\delta$-primary
hyperideals of $\Re.$ Assume that $a\circ m\in M-\phi(%
{\displaystyle\bigcup\limits_{i\in\Delta}}
M_{i}).$ This implies that $a\circ m\in M-\phi(M)$, for some $i\in\Delta.$ We
get either $a\in M$ or $m\in\delta(M),$ because of $M$ is a $\phi$-$\delta
$-primary. If $a\in M,$ then clearly we have $a\in%
{\displaystyle\bigcup\limits_{i\in\Delta}}
M_{i}. $If $m\in\delta(M)$, then we have $m\in\delta(%
{\displaystyle\bigcup\limits_{i\in\Delta}}
M_{i})$, since $M\subseteq%
{\displaystyle\bigcup\limits_{i\in\Delta}}
M_{i}.$ Hence $%
{\displaystyle\bigcup\limits_{i\in\Delta}}
M_{i}$ is a $\phi$-$\delta$-primary hyperideal.
\end{proof}

In the following, we give a characterization for $\phi$-$\delta$-primary hyperideals such
that $\phi$ is a hyperideal reduction and $\delta$ is a hyperideal expansion.

\begin{theorem}
\label{tm1} Let $\Re$ be a Krasner hyperring and $N$ be a proper hyperideal
of $\Re.$ Then the following statements hold:

(i)\ $N $ is a $\phi$-$\delta$-primary hyperideal;

(ii)\ For each $a\in\Re-\delta(N), (N:a)=N\cup(\phi(N):a);$

(iii)\ For each $a\in\Re-\delta(N), (N:a)=N$ or $(N:a)=(\phi(N):a);$

(iv)\ For each hyperideal $K,L $of $\Re, K\circ L\subseteq N$ and $K\circ
L\nsubseteq\phi(N)$ imply that $K\subseteq N$ or $L\subseteq\delta(N).$

(v)\ \label{thv} For each hyperideal $M $of $\Re$ such that $M\nsubseteq
\delta(N),$ then $(N:M)=N$ or $(N:M)=(\phi(N):M).$
\end{theorem}

\begin{proof}
$(i)\Rightarrow(ii): $ Suppose that $N$ is a $\phi$-$\delta$-primary
hyperideal and $a\in\Re-\delta(N).$ It is clear that $N\cup(\phi
(N):a))\subseteq(N:a). $Let $m\in(N:a). $Then we have $a\circ m\in N.$ If
$a\circ m\in\phi(N),$ then we obtain $m\in(\phi(N):a)\subseteq N\cup
(\phi(N):a)).$ Assume that $a\circ m\notin\phi(N).$ Since $a\circ m\in
N-\phi(N) $ and $a\notin\delta(N)$, then we get $m\in N\subseteq N\cup
(\phi(N):a)). $Hence, $(N:a)=N\cup(\phi(N):a)).$

$(ii)\Rightarrow(iii): $ It follows from the fact that a hyperideal is a union of
two hyperideals. Then it must be equal to one of them.

$(iii)\Rightarrow(iv):$ Let $K\circ L\subseteq N$, for some hyperideals $K$
and $L$ of $\Re.$ Assume that $L\nsubseteq\delta(N)$ and $K\nsubseteq N.$ Then there exists $m\in
L-\delta(N).$ We
have to show $K\circ L\subseteq\phi(N).$ By (iii), we have either $(N:m)=N$ or $(N:m)=(\phi
(N):m).$ If $K\circ m\subseteq N,$ then by (iii), we get $K\subseteq
(N:m)=N.$ For $a\in K-N$, it means $a\in(N:m)-N.$ Hence by part (iii),
$(N:m)=(\phi(N):m).$ Thus $K\subseteq(N:m)=(\phi(N):m)$ implies that $K\circ
m\subseteq\phi(N).$ On the other side, suppose that $m\in\delta(N).$ Then
$m\in L\cup\delta(N).$ Choose an element $m^{\prime}\in L-\delta(N), $so
$m\oplus m^{\prime}\subseteq L-\delta(N).$ Hence $K\circ m^{\prime}%
\subseteq\phi(N)$ and $K\circ(m\oplus m^{\prime})\subseteq\phi(N).$ Let $a\in
K$. Then $a\circ(m\oplus m^{\prime})\ominus a\circ m^{\prime}=a\circ
m\subseteq\phi(N).$ Thus $K\circ m\subseteq\phi(N).$ Therefore $K\circ
L\subseteq\phi(N).$

$(iv)\Rightarrow(v):$ Suppose that $M$ is a hyperideal of $\Re$ such that
$M\nsubseteq\delta(N).$ Also, note that $M\circ(N:M)\subseteq N.$ If
$M\circ(N:M)\subseteq\phi(N),$ then we have $(N:M)\subseteq(\phi
(N):M)\subseteq(N:M).$ Assume that $M\circ(N:M)\nsubseteq\phi(N).$ Then
by (iv), $(N:M)\subseteq N\subseteq(N:M).$

$(v)\Rightarrow(i):$ Suppose that $a\circ m\in N-\phi(N)$, with
$a\notin\delta(N),$ for $a,m\in\Re.$ Put $\Re\circ a=M$ and note that
$m\in(N:M).$ Then by (v), we get either $m\in(N:M)=N$ or $m\in
(N:M)=(\phi(N):M).$ The latter case is impossible, since $a\circ m\notin
\phi(N).$ Therefore, $N$ is a $\phi$-$\delta$-primary hyperideal.
\end{proof}

\begin{theorem}
(i)\ Let $T$ be a $\phi$-$\delta$-primary hyperideal of $\Re$ such that
$(\phi(T):a)=\phi(T:a)$, for each $a\in\Re.$ Then $(T:a)$ is a $\phi
$-$\delta$-primary hyperideal of $\Re.$

(ii)\ Suppose that $\delta\leq\delta_{1}$ are hyperideal expansions and
$T$ is a $\phi$-$\delta$-primary hyperideal of $\Re$ such that $\delta
_{1}(\phi(T))\subseteq\delta(T).$ Then $\delta(T)=\delta_{1}(T).$
\end{theorem}

\begin{proof}
$(i)\ $Let $T$ be a $\phi$-$\delta$-primary hyperideal of $\Re$ such that
$(\phi(T):a)=\phi(T:a)$, for each $a\in\Re.$ We show that $(T:a)$ is a
$\phi$-$\delta$-primary hyperideal of $\Re.$ For this, let we take $x,m\in
\Re$ such that $x\circ m\in(T:a)-\phi(T:a).$ Then we have $x\circ a\circ
m\in T.$ Since $(\phi(T):a)=\phi(T:a),$ then we also have $x\circ a\circ
m\notin\phi(T).$ Since $T$ is a $\phi$-$\delta$-primary hyperideal, then we get
either $x\in\delta(T)$ or $a\circ m\in T $ implying either $x\in\delta(T:a)$
or $m\in(T:a).$ Therefore, $(T:a) $ is a $\phi$-$\delta$-primary hyperideal
of $\Re.$

$(ii)$ Since $\delta\leq\delta_{1}, $ then we have $\delta(T)\subseteq\sqrt
{T}=\delta_{1}(T).$ For the converse, take $x\in\sqrt{T}.$ Then there exists
minimal $k\in%
\mathbb{N}
$ such that $x^{k}\in T,$ that is, $x^{k-1}\notin T.$ If $k=1,$ then we have
$x\in T\subseteq\delta(T).$ Now, assume that $k>1.$ We have two cases. Case
1\textbf{:} Let $x\in\delta_{1}(\phi(T))=\sqrt{\phi(T)}.$ Then we have
$x\in\delta(T)$. Case 2: Let $x\notin\sqrt{\phi(T)}.$ Then we have
$x^{k}\notin\phi(T),$ that is, $x^{k-1}\circ(x\circ\Re)\subseteq T$ and
$x\circ(x^{k-1}\circ\Re)\nsubseteq\phi(T).$ Since $T$ is a $\phi$-$\delta
$-primary hyperideal of $\Re,$ then by the Theorem \ref{tm1}, $x\in\delta(T)$ or
$x^{k-1}\circ\Re\subseteq T.$ The latter case is impossible, since
$x^{k-1}\notin T.$ Thus in both cases, we have $x\in\delta(T).$ Therefore,
$\delta(T)=\delta_{1}(T).$
\end{proof}

\begin{theorem}
Let $T$ be a $\phi$-$\delta$-primary hyperideal of $\Re$ such that
$\delta(T)\circ T\nsubseteq\phi(T).$ Then $T$ is a $\delta$-primary
hyperideal of $\Re.$
\end{theorem}

\begin{proof}
Let $a\circ m\in T $, for some $a,m\in\Re.$ If $a\circ m\notin\phi(T),$ then we
conclude either $a\in\delta(T)$ or $m\in T$ as $T$ is an $\phi$-$\delta
$-primary hyperideal of $\Re$. Assume that $a\circ m\in\phi(T).$ If
$a\circ T\nsubseteq\phi(T),$ then there exists $n\in T\ $such that $a\circ
n\notin\phi(T).$ Thus we have $a\circ(m\oplus n)\subseteq T-\phi(T),$ which
implies either $a\in\delta(T)$ or $m\oplus n\subseteq T.$ Then we get
$a\in\delta(T)$ or $m\in T,$ which completes the proof. Assume that
$a\circ T\subseteq\phi(T).$ Similarly, we may assume that $\delta(T)\circ
m\subseteq\phi(T).$ As $\delta(T)\circ T\nsubseteq\phi(T),$ we can find
$b\in\delta(T)$ and $m^{\prime}\in T$ such that $b\circ m^{\prime}\notin
\phi(T).$ Then we conclude that $(a\oplus b)\circ(m\oplus m^{\prime
})\subseteq T-\phi(T).$ Since $T$ is a $\phi$-$\delta$-primary hyperideal of
$\Re$, then we have either $a\oplus b\subseteq\delta(T)$ or $m\oplus m^{\prime
}\subseteq T,$ which implies that $a\in\delta(T)$ or $m\in T.$ Therefore,
$T$ is a $\delta$-primary hyperideal of $\Re.$
\end{proof}

\begin{definition}
i) \cite{18} A hyperideal expansion $\delta$ is said to be global if for any
hyperring good homomorphism $\mu:\Re\rightarrow\hat{S},$ $\delta(\mu
^{-1}(M))=\mu^{-1}(\delta(M))$, for each $M\in L(\hat{S}).$

ii) A hyperideal reduction $\phi$ is said to be a \textit{global} if for any
homomorphism $\mu:\Re\rightarrow\hat{S}, \phi(\mu^{-1}(M))=\mu^{-1}%
(\phi(M))$, for each $M\in L(\hat{S}).$
\end{definition}

For instance, the hyperideal reductions $\phi_{0}, \phi_{1}$ and the
hyperideal expansions $\delta_{0}, \delta_{1}$ are both \textit{global.}

\begin{theorem}
\label{thom} Let $\mu$ be a good homomorphism from Krasner hyperring
$(\Re,\oplus,\circ)$ into a Krasner hyperring $(\hat{S},+,\cdot)$. The following statements hold:

(i)\ Let $\mu:\Re\rightarrow\hat{S}$ be a good homomorphism and $M$ be a
$\phi$-$\delta$-primary hyperideal of $\hat{S}$ such that $\phi$ is global
reduction function and $\delta$ is global expansion function. Then $\mu
^{-1}(M)=\Re$ or $\mu^{-1}(M)$ is a $\phi$-$\delta$-primary hyperideal of
$\Re.$

(ii) Let $\mu:\Re\rightarrow\hat{S}$ be a good epimorphism and $N$ be a
hyperideal of $\Re$ containing $Ker(\mu).$ Suppose that $\phi$ is global
reduction function and $\delta$ is global expansion function. Then $N$ is a
$\phi$-$\delta$-primary hyperideal of $\Re$ if and only if $\mu(N)$ is a
$\phi$-$\delta$-primary hyperideal of $\hat{S}.$
\end{theorem}

\begin{proof}
$(i)$\ Let $\mu:\Re\rightarrow\hat{S}$ be a hyperring homomorphism and $M$ be
a $\phi$-$\delta$-primary hyperideal of $\hat{S}$ such that $\mu^{-1}%
(M)\neq\Re.$ Assume $a\circ m\in\mu^{-1}(M)-\phi(\mu^{-1}(M)),$ for some
$a,m\in\Re.$ Since $\phi$ is global, then $\phi(\mu^{-1}(M))=\mu^{-1}(\phi
(M))$ and thus $a\circ m\in\mu^{-1}(M)-\mu^{-1}(\phi(M)).$ This implies that
$\mu(a\circ m)=\mu(a)\cdot\mu(m)\in M-\phi(M).$ Since $M$ is a $\phi
$-$\delta$-primary hyperideal of $\hat{S},$ then we have either $\mu
(a)\in\delta(M)$ or $\mu(m)\in M.$ Since $M\subseteq\mu^{-1}(M)$, then we get
$a\in\mu^{-1}(\delta(M))=\delta(\mu^{-1}(M))$ or $m\in\mu^{-1}(M).$%
Therefore, $\mu^{-1}(M)$ is a $\phi$-$\delta$-primary hyperideal.

$(ii)$\ Let us suppose that $\mu(N)$ is a $\phi$-$\delta$-primary hyperideal of
$\hat{S},$ where $N$ is a hyperideal of $\Re$ containing $Ker(\mu).$ By $(i)$, $\mu^{-1}(\mu(N))=N$ is a $\phi$-$\delta$-primary hyperideal of
$\Re.$ For the converse, let $N$ be a $\phi$-$\delta$-primary hyperideal of
$\Re$ and $a^{\prime}\cdot m^{\prime}\in\mu(N)-\phi(\mu(N))$, for some
$a^{\prime}, m^{\prime}\in\hat{S}$. Then $\exists a,m\in\Re$ with
$\mu(a)=a^{\prime}$ and $\mu(m)=m^{\prime},$ since $\mu$ is surjective.
$a^{\prime}\cdot m^{\prime}=\mu(a)\cdot\mu(m)=\mu(a\circ m)=\mu(x)\in
\mu(N)-\phi(\mu(N))$, for some $x\in\Re.$ So $0\in\mu(a\circ m)-\mu
(x)=\mu(a\circ m\ominus x).$ Hence there exists $t\in a\circ m\ominus x$ such
that $\mu(t)=0_{\hat{S}}.$ $a\circ m\in t\oplus x\subseteq Ker(\mu)+N\subseteq
N+N\subseteq N.$ Then $\phi$ is global and $Ker(\mu)\subseteq N,$ we have
$\phi(\mu(N))=\mu(\phi(N))$ and also $a\circ m\in N-\phi(N).$ Since $N$ is a
$\phi$-$\delta$-primary hyperideal of $\Re,$ then we have either $a\in
\delta(N)$ or $m\in N.$ Since $N\subseteq\mu(N),$ then we have $a^{\prime}%
=\mu(a)\in\delta(\mu(N))$ or $m^{\prime}=\mu(m)\in\mu(N).$ Therefore,
$\mu(N)$ is a $\phi$-$\delta$-primary hyperideal of $\hat{S}.$
\end{proof}

As an instant consequence of the previous theorem, we get the following
explicit results.

\begin{corollary}
\label{cfac}(i) Let $N$ be a $\phi$-$\delta$-primary hyperideal of $\Re$ and
$M$ be a hyperideal of $\Re$ with $M\nsubseteq N.$ Suppose that $\phi$ is
global reduction function and $\delta$ is global expansion function. Then
$N\cap M$ is a $\phi$-$\delta$-primary hyperideal of $M.$

(ii) Let $N$ and $M$ be two hyperideals of $\Re$ with $M\subseteq N$. Suppose
that $\phi$ is global reduction function and $\delta$ is a global expansion
function. Then $N$ is a $\phi$-$\delta$-primary hyperideal of $\Re$ if and
only if $N/M$ is a $\phi$-$\delta$-primary hyperideal of $\Re/M.$
\end{corollary}

\begin{theorem}
\label{tweak} Let $N$ be a hyperideal of $\Re.$ Then $N$ is a $\phi
$-$\delta$-primary hyperideal of $\Re$ if and only if $N/\phi(N)$ is a
weakly $\delta$-primary hyperideal of $\Re/\phi(N).$
\end{theorem}

\begin{proof}
Assume that $N$ is a $\phi$-$\delta$-primary hyperideal of $\Re.$ Let
$(a\oplus\phi(N))\circ(m\oplus\phi(N))\subseteq N/\phi(N)-\{0_{\Re/\phi
(N)}\}.$ Then we have $a\circ m\in N-\phi(N).$ Since $N$ is a $\phi$-$\delta
$-primary hyperideal of $\Re,$ then we get $a\in\delta(N)$ or $m\in N,$ which
implies $a\oplus\phi(N)\in\delta(N/\phi(N))$ or $m\oplus\phi(N)\subseteq
N/\phi(N).$ Hence, $N/\phi(N)$ is a weakly $\delta$-primary hyperideal of
$\Re/\phi(N).$ For the converse, let $a,m\in\Re$ such that $a\circ m\in
N-\phi(N).$ This implies that $(a\oplus\phi(N))\circ(m\oplus\phi(N))\subseteq
N/\phi(N)-\{0_{\Re/\phi(N)}\}.$ Since $N/\phi(N)$ is a weakly $\delta$-primary
hyperideal of $\Re/\phi(N),$ then we get $a\oplus\phi(N)\in\delta(N/\phi(N))$ or
$m\oplus\phi(N)\subseteq N/\phi(N)$ implying $a\in\delta(N)$ or $m\in
N.$ Therefore, $N$ is a $\phi$-$\delta$-primary hyperideal of $\Re.$
\end{proof}

Let $\Re$ be a commutative Krasner hyperring, $S\subseteq\Re$ be a
multiplicatively closed subset of $\Re$ and $T$ be a proper hyperideal of
$\Re.$

\begin{proposition}
\label{ploc} Let $\phi_{S}:L(\Re
_{S})\rightarrow L(\Re_{S})\cup\{\emptyset\}$ be a hyperideal reduction
function and $\delta_{S}:L(\Re_{S})\rightarrow L(\Re_{S})$ be a hyperideal
expansion function such that $\delta_{S}(N_{S})=\delta(N)_{S}$, for each $N\in
L(\Re).$ If $T$ is a $\phi$-$\delta$-primary hyperideal such that
$T\cap S=\emptyset$ and $\phi(T)_{S}\subseteq\phi_{S}(T_{S}),$ then $T_{S}$
is a $\phi_{S}$-$\delta_{S}$-primary of $\Re_{S}.$
\end{proposition}

\begin{proof}
Let $\frac{a}{s}\circ\frac{m}{t}\in T_{S}-\phi_{S}(T_{S})$, for some $a,m\in
\Re; s,t\in S.$ Then we have $p\circ a\circ m\in T-\phi(T)$, for some $p\in
S$, since $\phi(T)_{S}\subseteq\phi_{S}(T_{S}).$ Since $T$ is a $\phi
$-$\delta$-primary hyperideal of $\Re,$ then we have either $a\in\delta(T)$ or
$p\circ m\in T.$ If $a\in\delta(T),$ then $\frac{a}{s}\in\delta
(T)_{S}=\delta_{S}(T_{S})$. If $p\circ m\in T,$ then we have $\frac{m}%
{t}=\frac{p\circ m}{p\circ t}\in T_{S}.$ Therefore, $T_{S}$ is a $\phi_{S}%
$-$\delta_{S}$-primary hyperideal of $\Re_{S}.$
\end{proof}

Let $\varphi_{i}:L(\Re_{i}%
)\rightarrow L(\Re_{i})\cup\{\emptyset\}$ be hyperideal reduction functions
and $\gamma_{i}:L(\Re_{i})\rightarrow L(\Re_{i})$ be hyperideal expansion
functions for $i=1,2$. Suppose that $\Re=\Re_{1}\times\Re_{2}.$ Also, each
hyperideal $N$ of $\Re$ has the form $N=N_{1}\times N_{2},$ where $N_{i}%
$ is a hyperideal of $\Re_{i}.$ Furthermore, $\phi:L(\Re)\rightarrow
L(\Re)\cup\{\emptyset\}$ defined by $\phi(N_{1}\times N_{2})=\varphi
_{1}(N_{1})\times\varphi_{2}(N_{2})$ is a hyperideal reduction function and
also the function $\delta:L(\Re)\rightarrow L(\Re)$ defined by $\delta
(N_{1}\times N_{2})=\gamma_{1}(N_{1})\times\gamma_{2}(N_{2})$ is a hyperideal
expansion function.

\begin{proposition}
\label{propcart} Let the notation be as in the above pharagraph and
$N=N_{1}\times N_{2}.$ Then each of the types of $N$ are $\phi$-$\delta
$-primary hyperideals of $\Re=\Re_{1}\times\Re_{2}.$

(i) $N=N_{1}\times N_{2},$ where $N_{i}$ is a proper hyperideal of $\Re
_{i}$ with $\varphi_{i}(N_{i})=N_{i}.$

(ii) $N=N_{1}\times\Re_{2},$ where $N_{1}$ is $\gamma_{1}$-primary
hyperideal of $\Re_{1}$.

(iii) $N=\Re_{1}\times N_{2},$ where $N_{2}$ is $\gamma_{2}$-primary
hyperideal of $\Re_{2}.$

(iv) $N=N_{1}\times\Re_{2},$ where $N_{1}$ is $\varphi_{1}$-$\gamma_{1}%
$-primary hyperideal of $\Re_{1}$ and $\varphi_{2}(\Re_{2})=\Re_{2}$.

(v) $N=\Re_{1}\times N_{2},$ where $N_{2}\ $is $\varphi_{2}$-$\gamma_{2}%
$-primary hyperideal of $\Re_{2}$ and $\varphi_{1}(\Re_{1})=\Re_{1}$.
\end{proposition}

\begin{proof}
(i) Obviously $N$ is $\phi$-$\delta$-primary hyperideal, since $N_{1}\times
N_{2}-\phi(N_{1}\times N_{2})=\oslash$.

(ii) Let $N=N_{1}\times\Re_{2}$, where $N_{1}\ $is $\gamma_{1}$-primary
hyperideal of $\Re_{1}$. We can see that $N_{1}\times\Re_{2}$ is $\delta
$-primary hyperideal of $\Re_{1}$. By the Proposition \ref{p1}, $N$ is $\phi
$-$\delta$-primary hyperideal of $\Re.$

(iii) The proof is similar to (ii).

(iv) Suppose that $(a_{1},a_{2})\circ(m_{1},m_{2})\in N_{1}\times\Re_{2}%
-\phi(N_{1}\times\Re_{2})=\varphi_{1}(N_{1})\times\varphi_{2}(\Re_{2}).$ Then
$(a_{1}\circ m_{1},a_{2}\circ m_{2})\in(N_{1}-\varphi_{1}(N_{1}))\times
(\Re_{2}-\varphi_{2}(\Re_{2}).$ Since $N_{1}\ $is $\varphi_{1}$-$\gamma_{1}%
$-primary hyperideal of $\Re_{1},$ then we get either $a_{1}\in N_{1}$ or $m_{1}%
\in\gamma_{1}(N_{1}).$ So $(a_{1},a_{2})\in N_{1}\times\Re_{2}$ or
$(m_{1},m_{2})\in\gamma_{1}(N_{1})\times\Re_{2}\subseteq\gamma_{1}%
(N_{1})\times\gamma_{2}(\Re_{2}).$ It means $N$ is a $\phi$-$\delta$-primary
hyperideal of $\Re.$

(v) Proof is similar to (iv).
\end{proof}

\begin{theorem}
\label{thcart} Let the notation be as in the Proposition \ref{propcart},
$N=N_{1}\times N_{2}$, where $\varphi_{i}(N_{i})\neq N_{i}.$ Then $N$ is a
$\phi$-$\delta$-primary hyperideal of $\Re$ if and only if $N$ is one of the
following types:

(i) $N=N_{1}\times\Re_{2},$ where $N_{1}\ $is $\varphi_{1}$-$\gamma_{1}%
$-primary hyperideal of $\Re_{1}$ that should be $\gamma_{1}$-primary if
$\varphi_{2}(\Re_{2})\not =\Re_{2}$.

(ii) $N=\Re_{1}\times N_{2},$ where $N_{2}$ is $\varphi_{2}$-$\gamma_{2}%
$-primary hyperideal of $\Re_{2}$ that should be $\gamma_{2}$-primary if
$\varphi_{1}(\Re_{1})\not =\Re_{1}$.
\end{theorem}

\begin{proof}
$(\Longleftarrow)$ It follows from the Proposition \ref{propcart}.

$(\Longrightarrow)$ Suppose that $N$ is a $\phi$-$\delta$-primary hyperideal
of $\Re$, where $\varphi_{i}(N_{i})\neq N_{i}.$ Let $a\circ m\in N_{1}%
-\varphi_{1}(N_{1})$, for some $a,m\in\Re_{1}.$ Thus $(a,0)\circ(m,0)=(a\circ
m,0)\in N-\phi(N).$ Since $N$ is a $\phi$-$\delta$-primary hyperideal of
$\Re,$ then $(a,0)\in N$ or $(m,0)\in\delta(N).$ So $a\in N_{1}$ or
$m\in\gamma_{1}(N_{1}).$ Therefore $N_{1}\ $is $\varphi_{1}$-$\gamma_{1}%
$-primary hyperideal of $\Re_{1}.$ Similarly, we can find $N_{2}$ is
$\varphi_{2}$-$\gamma_{2}$-primary hyperideal of $\Re_{2}.$ We have to
show now that $N_{1}=\Re_{1}$ or $N_{2}=\Re_{2}.$ Suppose that $N_{2}\not =\Re
_{2}.$ Let we take $m_{1}\in N_{1}-\varphi_{1}(N_{1}),$ $m_{2}\in\Re_{2}-N_{2}.$
Notice that $(1,0)\circ(m_{1},m_{2})=(m_{1},0)\in N-\phi(N).$ This implies
that $(1,0)\in\delta(N),$ so we find $1\in\gamma_{1}(N_{1}).$ Then $N_{1}%
=\Re_{1}.$ Similarly, one can easily find $N_{2}=\Re_{2}$, if $N_{1}\not =%
\Re_{1}$. Without loss of generality $N_{1}\not =\Re_{1}$. Let we show now that
$N_{1}$ is $\gamma_{1}$-primary hyperideal with $\varphi_{2}(\Re_{2})\neq
\Re_{2}$. For $m^{\prime}\in\Re_{2}-\varphi_{2}(\Re_{2})$, let $x\circ m\in
N_{1}$, for some $x,m\in\Re_{1}$. We have that $(x,1)\circ
(m,m^{\prime})=(x\circ m,m^{\prime})\in N-\phi(N)$. Since $N$ is a $\phi
$-$\delta$-primary hyperideal of $\Re$, then we find $(x,1)\in\delta
(N)=\gamma_{1}(N_{1})\times\gamma_{2}(N_{2})$ or $(m,m^{\prime})\in N$ which
implies that $x\in\gamma_{1}(N_{1}))$ or $m\in N_{1}$. Therefore $N_{1}$ is a
$\gamma_{1}$-primary hyperideal of $\Re_{1}.$ If $\varphi_{1}(\Re_{1})\neq
\Re_{1}$ and $N_{1}=\Re_{1}$, then similarly one can prove that $N_{2}$ is a
$\gamma_{2}$-primary hyperideal of $\Re_{2}$.
\end{proof}

\section{Conclusion}

In this paper, generalizations of prime and primary hyperideals were given
using the function $\phi.$ We introduced $\phi$-prime, $\phi$-primary and
$\phi$-$\delta$-primary hyperideals and we gave several characterizations to
classify them. We investigated many properties of $\phi$-prime, $\phi$-primary
and $\phi$-$\delta$-primary hyperideals under particular cases. In the future
work, one can develop the study of $\phi$-$\delta$-primary hyperideals.\\
\newline

{\textbf{Compliance with Ethical Standards}}\\

\textbf{Conict of Interest} All authors declare that they
have no conict of interest.\\

\textbf{Ethical approval} This article does not contain any
studies with human participants or animals performed
by any of the authors.

\end{document}